\documentclass[10pt]{article}
\usepackage{amsmath,amsthm,amsfonts,amssymb}
\usepackage{enumerate}
\usepackage{hyperref}
\usepackage[a4paper,includeheadfoot,margin=2.54cm]{geometry}

\newtheorem{theorem}{Theorem}
\newtheorem{lemma}[theorem]{Lemma}
\newtheorem{conjecture}[theorem]{Conjecture}

\newtheorem{corollary}[theorem]{Corollary}
\theoremstyle{remark}
\newtheorem{remark}[theorem]{Remark}
\theoremstyle{definition}

\newcommand{\scrE}{\mathcal{E}}

\newcommand{\scrP}{\mathcal{P}}
\newcommand{\scrQ}{\mathcal{Q}}

\newcommand{\EE}{\mathbb{E}}
\newcommand{\FF}{\mathbb{F}}

\newcommand{\RR}{\mathbb{R}}

\newcommand{\AG}{\text{AG}}
\newcommand{\PGL}{\text{PGL}}

\newcommand{\gauss}[2]{\genfrac{[}{]}{0pt}{}{#1}{#2}}

\title{Remarks on the Erd\H{o}s Matching Conjecture for Vector Spaces}
\author{Ferdinand Ihringer\thanks{Department of Mathematics: Analysis, Logic and Discrete Mathematics, Ghent University, Belgium, \href{mailto:ferdinand.ihringer@ugent.be}{ferdinand.ihringer@ugent.be}.
}}

\begin{document}

\maketitle

\begin{abstract}
In 1965, Paul Erd\H{o}s asked about the largest family $Y$ of $k$-sets in $\{ 1, \ldots, n \}$ such that
$Y$ does not contain $s+1$ pairwise disjoint sets. This problem is commonly known as the Erd\H{o}s Matching
Conjecture. We investigate the $q$-analog of this question, that is 
we want to determine the size of a largest family $Y$ of $k$-spaces in $\mathbb{F}_q^n$
such that $Y$ does not contain $s+1$ pairwise disjoint $k$-spaces.
Here we call two subspaces disjoint if they intersect trivially.

Our main result is, slightly simplified, that if $16 s \leq \min\{  q^{\frac{n-k}{4}},$ $q^{\frac{n-2k+1}{3}} \}$, 
then $Y$ is either small or a union of intersecting families. Thus we show the
Erd\H{os} Matching Conjecture for this range. The proof uses a method due to Metsch.
We also discuss constructions. In particular, we show that for larger $s$, 
there are large examples which are close in size to a union of intersecting
families, but structurally different.

As an application, we discuss the close relationship between the Erd\H{o}s Matching Conjecture 
for vector spaces and Cameron-Liebler line classes (and their generalization to $k$-spaces), a popular
topic in finite geometry for the last 30 years. More specifically, we propose the Erd\H{o}s Matching Conjecture (for vector spaces)
as an interesting variation of the classical research on Cameron-Liebler line classes.
\end{abstract}

\section{Introduction}

In 1961, Erd\H{o}s, Ko, and Rado famously showed that an \textit{intersecting family} of $k$-sets in $\{ 1, \ldots, n \}$
has at most size $\binom{n-1}{k-1}$ and, if $n > 2k$, consists of all $k$-sets which contain a fixed element in the case of equality \cite{ErdHos1961}.
Hence, intersecting families are families of $k$-sets with no $2$ of its elements pairwise disjoint and we know the largest such families. 
If we replace $2$ by a parameter $s+1$, then we obtain the setting of the Erd\H{o}s Matching Conjecture from 1965 \cite{Erdos1965}. Let us say that a family without $s+1$ pairwise disjoint elements is an $s$-EM-family.
There are two natural choices for $s$-EM-families of $k$-sets in $\{ 1, \ldots, n \}$.
The first one, let us call it $Y_1$, is the family of $k$-sets which intersect $\{ 1, \ldots, s \}$
non-trivially. The family $Y_1$ has size $\binom{n}{k} - \binom{n-s}{k}$. 
The second one, let us call it $Y_2$, is the family of $k$-sets which are contained in $\{ 1, \ldots, k(s+1)-1\}$.
The family $Y_2$ has size $\binom{k(s+1)-1}{k}$.
Erd\H{o}s states in \cite{Erdos1965} that the following ``is not impossible'':
\begin{conjecture}[The Erd\H{o}s Matching Conjecture]\label{conj:em1}
  Let $Y$ be a largest $s$-EM-family of $k$-sets of $\{ 1, \ldots, n \}$. 
  Then $|Y| = \max\{ |Y_1|, |Y_2| \}$.
\end{conjecture}
The conjecture was proven for $k=2$ by Erd\H{o}s and Gallai \cite{Erdos1959} 
and for $k=3$ by Frankl \cite{Frankl2017b}.
In particular, Frankl showed the conjecture for $n \geq (2s+1)k-s$ \cite{Frankl2013a}
and for $n \leq (s+1)(k+\epsilon)$ where $\epsilon$ depends on $k$ \cite{Frankl2017a}.
Furthermore, Frankl and Kupavskii cover $n \geq \frac53 sk - \frac23 s$ for sufficiently large $s$ \cite{Frankl2018}.
A more complete overview on the history of the problem can be found in \cite{Frankl2018}.

For our purposes, let us state the Erd\H{o}s Matching Conjecture in a way that makes
it more generic, easily transferable between lattices, and includes a structural classification.
\begin{conjecture}[The Erd\H{o}s Matching Conjecture (variant)]\label{conj:em2}
  Let $Y$ be a largest $s$-EM-family of $k$-sets of $\{ 1, \ldots, n \}$.
  Then $Y$ is the union of $s$ intersecting families or its complement.
\end{conjecture}
Note that one can deduce Conjecture \ref{conj:em1} from Conjecture \ref{conj:em2}
due to the fact that the structure of large intersecting families of $k$-sets is well-known.
In this paper we consider $s$-EM-families of $k$-spaces in $\FF_q^n$.
We say that two subspaces are \textit{disjoint} if their intersection is the trivial subspace.
The natural conjecture here is as follows.
\begin{conjecture}\label{conj:em_vs}
  Let $Y$ be a largest $s$-EM-family of $k$-spaces of $\FF_q^n$. 
  Then $Y$ is the union of $s$ intersecting families or its complement.
\end{conjecture}

We consider the setting in vector spaces as particularly interesting:
In the set case, we have that if $k$ divides $n$ and $Z$ is a family of $k$-sets which partitions $\{ 1, \ldots, n \}$,
then $Z$ intersects an $s$-EM-family $Y$ in at most $s$ elements. It is not hard to see that this implies
\begin{align*}
  |Y| \leq s\binom{n-1}{k-1}.
\end{align*}
One can show that equality in this bound only holds when $Y$ is, 
in the language of \cite{Godsil1993}, a certain type of equitable bipartition of the Johnson graph
or, in the language of \cite{Filmus2018}, a Boolean degree $1$ function of the Johnson graph.
These do not exist except for $s=0, 1, \frac{n}{k}-1, \frac{n}{k}$, so the bound above can 
be instantaneously improved by one.

Write $\gauss{n}{k}_q$ for the Gaussian (or $q$-binomial) coefficient. For $n$ and $k$ integers and $q$ a 
prime power, $\gauss{n}{k}_q$ is the 
number of $k$-spaces in $\FF_q^n$.
In the vector space analog, if $k$ divides $n$ and $Z$ is a family of $k$-spaces which 
partitions $\FF_q^n \setminus \{ 0 \}$, so a \textit{spread} of $\FF_q^n$, then the same behavior occurs. In this setting,
Boolean degree 1 functions are known as Cameron-Liebler classes of $k$-spaces \cite{Blokhuis2018,Filmus2018}.
Here we have the analogous result, that is a $s$-EM-family $Y$ of $k$-spaces intersects $Z$ in 
at most $s$ elements, from which it follows that
\begin{align*}
  |Y| \leq s \gauss{n-1}{k-1}_q.
\end{align*}
It is easy to find trivial examples for Cameron-Liebler classes which meet this bound for small $s$, but the general picture is not clear.
Throughout the paper, we use projective notation and call $1$-spaces \textit{points}, $2$-spaces \textit{lines}, 
$3$-spaces \textit{planes}, and $(n-1)$-spaces \textit{hyperplanes}.
The trivial examples, up to taking complements and besides the empty set, are all $k$-spaces through a fixed point, all $k$-spaces in a fixed hyperplane,
and the disjoint union of the first two examples.
Non-trivial Cameron-Liebler classes appear to exist for $(n, k) = (4, 2)$ and any $q > 2$ \cite{Bruen1999,DeBeule2016,Drudge1998,Feng2015,
Gavrilyuk2018,Rodgers2012}, but 
not for $n \geq 2k$ when $n > 4$. The latter is at least true for $q \in \{ 2, 3, 4, 5 \}$ \cite{Filmus2018}.
The fact that non-trivial examples exist for $(n, k) = (4, 2)$ does not imply that the Erd\H{o}s Matching
Conjecture is false as these examples might have $s+1$ pairwise disjoint elements which do not extend
to a spread of $\FF_q^n$. Indeed, all known non-trivial examples investigated by the author are not $s$-EM families.
Nonetheless, it makes one doubt that Conjecture \ref{conj:em_vs} is true.

It is known that there are no non-trivial small examples for Cameron-Liebler classes.
Metsch established a proof technique in \cite{Metsch2017} which essentially shows that small
Cameron-Liebler classes are $s$-EM-families. He used it to show the following.
\begin{theorem}[Metsch {\cite[Theorem 1.4]{Metsch2017}}]\label{thm:metsch_n_eq_2k}
 All Cameron-Liebler classes $Y$ of $k$-spaces in $\FF_q^{2k}$ with $5 \cdot |Y| \leq q \gauss{n-1}{k-1}$ are trivial.
\end{theorem}
Note that \cite{Metsch2017} states that $q$ has to be sufficiently large, but this condition can be dropped \cite{Ihringer2018}. Blokhuis, De Boeck and D'haeseleer generalized this to $k$-spaces in $\FF_q^n$ \cite[Theorem 4.9]{Blokhuis2018},
but the proof of their result (and therefore the stated result) contains a minor mistake which we amend with Theorem \ref{thm:main_cl}.

We investigate $s$-EM families of $k$-spaces in $\FF_q^n$.
Let $\ell$ be the integer satisfying $\frac{q^{\ell-1}-1}{q-1} < s \leq \frac{q^{\ell}-1}{q-1}$.
Write $n=mk+r$ with $0 \leq r < k$.
Our main result is as follows. 
\begin{theorem}\label{thm:main}
  Let $n \geq 2k$ and let $Y$ be a largest $s$-EM family of $k$-spaces in $\FF_q^n$.
  If $16 s \leq$ $\min\{$ $q^{\frac{n-k-\ell+2}{3}},$ $q^{\frac{n-k-r}{3}},$ $q^{\frac{n}2-k+1} \}$,
  then $Y$ is the union of $s$ intersecting families.
\end{theorem}
Note that we did not optimize the factor $16$ on the left hand side of the inequality. In fact, 
$16$ can be certainly replaced by a factor $c_q$ with $\lim_{q \rightarrow \infty} c_q = 1$.
Besides this, the argument is optimized to the best knowledge of the author. 
If we bound $2q^{\ell-1}$ instead of $s$, which we can as $s \leq \frac{q^\ell-1}{q-1} \leq \frac{q}{q-1} q^{\ell-1} \leq 2q^{\ell-1}$, then we see that $\ell \leq \lceil \frac{n-k+5}{4} \rceil \leq \frac{n-k+8}{4}$ suffices. Hence, $16s \leq \min\{  q^{\frac{n-k}{4}}, q^{\frac{n-2k+1}{3}}, q^{\frac{n}{2}-k+1} \}$. Here the last bound is redundant, thus we arrive at the simplified claim of the abstract. For $n \geq 3k-4$, this simplifies further to $16s \leq q^{\frac{n-k}{4}}$.

Cameron-Liebler classes are completely classified for 
$q \in \{ 2, 3, 4, 5 \}$ \cite{Drudge1998,Filmus2018,Gavrilyuk2018a,Gavrilyuk2014}, while in general only 
some limited characterizations are known.
For the special case of $(n,k) = (4,2)$ Gavrilyuk and Metsch \cite{Gavrilyuk2014}, and Metsch \cite{Metsch2014} showed
highly non-trivial existence conditions. The latter is as follows.
\begin{theorem}[{Metsch \cite{Metsch2014}}]
  Let $Y$ be a Cameron-Liebler class of lines in $\FF_q^4$ of size $s (q^2+q+1)$. 
  If $s \leq C q^{4/3} (q^2+q+1)$ for some universal constant $C$,
  then $s \leq 2$ and $Y$ is trivial.
\end{theorem}
From Theorem \ref{thm:main} we deduce the following.
\begin{theorem}\label{thm:main_cl}
  Let $n \geq 2k$ and let $Y$ be a Cameron-Liebler class of $k$-spaces in $\FF_q^n$ of size $s \gauss{n-1}{k-1}$.
  If $16 s \leq \min\{ q^{\frac{n-k-\ell+2}{3}}, q^{\frac{n-2k-\tilde{r}+1}{3}} \}$, where $n = \tilde{m} k - \tilde{r}$ with $0 \leq \tilde{r} < k$,
  then $s \leq 2$ and $Y$ is trivial.
\end{theorem}
Our original intent was to improve a result in \cite{Blokhuis2018} for certain choices of parameters, 
but as we discovered a mistake in the argument in \cite{Blokhuis2018}, this is the only such bound at the time of writing.\footnote{Our bound is $Cs \leq q^{\frac{n}{2}-k+1}$ for $n$ large enough while the alleged bound in \cite{Blokhuis2018} is $Cs \leq q^{\frac{n}{2}-k+\frac{1}{2}}$ and only holds for $n \geq 3k$. We consider the behavior for $n$ close to $2k$ as the most interesting.}
Note that the statement is still empty for $2k < n < \frac{5}{2} k$.

\section{Preliminaries}

\subsection{Gaussian Coefficients}

For any real numbers $a$ and $q$, we define $[a]_q := \lim_{r \rightarrow q} \frac{r^a-1}{r-1}$ and, for $b$ an integer, we define the Gaussian coefficient by
\begin{align*}
 \gauss{a}{b}_q = \begin{cases}
                   0 & \text{ if } b < 0,\\
                   \prod_{i=0}^{b-1} \frac{[a-i]_q}{[b-i]_q} & \text{ otherwise.}
                  \end{cases}
\end{align*}
We have $\gauss{a}{b}_1 = \binom{a}{b}$.
We write $[a]$ instead of $[a]_q$ and $\gauss{a}{b}$ instead of $\gauss{a}{b}_q$ as $q$ 
is usually fixed. Note that $\gauss{n}{k}$ corresponds to the number
of $k$-spaces in $\FF_q^n$.
The following can be derived from \cite[Lemma 34]{Ihringer2014b} 
(alternatively, \cite[Lemma 2.1]{Ihringer2015} for $q \geq 3$).
Note that while \cite{Ihringer2015} and \cite{Ihringer2014b} both assume
that $a$ is an integer and $q$ a prime power, the proofs there only use that $q \geq 2$.

\begin{lemma}\label{lem:gauss_bnds}
  Let $a \geq b \geq 0$ and $q \geq 2$. Then
  \begin{align*}
    &q^{b(a-b)} \leq \gauss{a}{b} \leq (1+5q^{-1}) q^{b(a-b)} \leq \frac{7}{2} q^{(b(a-b)} < 4 q^{b(a-b)}
    \intertext{ and, if $q \geq 4$, }
    &q^{b(a-b)} \leq \gauss{a}{b} \leq (1+2q^{-1}) q^{b(a-b)} \leq 2 q^{b(a-b)}.
  \end{align*}
\end{lemma}
We will the lemma without reference throughout the document, 
mostly for $\gauss{a}{b} \leq 4 q^{b(a-b)}$.
For $[a]$ we use the better bound of $[a] \leq \frac{q}{q-1} q^{a-1} \leq 2 q^{a-1}$.
The Gaussian coefficients satisfy the following generalization of Pascal's identity:
\begin{align}
  \gauss{a}{b} = q^b \gauss{a-1}{b} + \gauss{a-1}{b-1} = q^{a-b} \gauss{a-1}{b-1} + \gauss{a-1}{b}.   \label{eq:pascal}
\end{align}

This enables us to make the following useful observation.

\begin{lemma}\label{lem:high_order_cancellation}
 Let $q \geq 2$, $x$ an integer, $a \in \RR$ with $a \geq x \geq 1$, and $b$ an integer with $a \geq b \geq 2$.
 Then 
 \begin{align*}
    \gauss{a}{b} - q^{bx} \gauss{a-x}{b} &\leq \rho (1+\tfrac{1}{q-1}) q^{x+(b-1)(a-b)-1}.
  \intertext{Here, $\rho = 1 + 5q^{-1}$ for $q \in \{2,3\}$ and $\rho = 1+2q^{-1}$ otherwise. In particular,}
    \gauss{a}{b} - q^{bx} \gauss{a-x}{b} & \leq (1+12q^{-1}) q^{x+(b-1)(a-b)-1} \text{ if } q \geq 2, \text{ and }\\
    \gauss{a}{b} - q^{bx} \gauss{a-x}{b} & \leq \frac{3}{2} q^{x+(b-1)(a-b)-1} \text{ if } q \geq 7.
 \end{align*}
\end{lemma}
\begin{proof}
  Equation \eqref{eq:pascal} together with Lemma \ref{lem:gauss_bnds} implies that
  \begin{align*}
    \gauss{a}{b} = q^b \gauss{a-1}{b} + \gauss{a-1}{b-1} \leq q^b \gauss{a-1}{b} + \rho q^{(b-1)(a-b)}.
  \end{align*}
  If we repeat this $x$ times, then we obtain (we bound the geometric series by $\frac{q}{q-1}$)
  \begin{align*}
    \gauss{a}{b} \leq q^{bx} \gauss{a-x}{b} + \rho (1 + \frac{q}{q-1} \cdot q^{-1}) q^{x + (b-1)(a-b) - 1}.
  \end{align*}
  The assertion follows.
\end{proof}
% 
% We use this bound mostly for $x=b$ and $x=b+1$, so let use restate the bound for these particular cases:
% \begin{align*}
%   \gauss{a}{b} - q^{b^2} \gauss{a-b}{b} \leq (1+12q^{-1}) q^{(b-1)(a-b+1)}. %\label{eq:diff_bnd_1}
%   \intertext{ and }
%   \gauss{a}{b} - q^{b^2+b} \gauss{a-b-1}{b} \leq (1+12q^{-1}) q^{1+(b-1)(a-b+1)}. %\label{eq:diff_bnd_2}
% \end{align*}
% % And this trivial consequence:
% % \begin{align}
% %   \gauss{a}{b} - q^{b^2+b} \gauss{a-b-1}{b} - q^{1+(b-1)(a-b+1)} \leq 7q^{(b-1)(a-b+1)}.\label{eq:diff_bnd_2}
% % \end{align}

\begin{remark}
\begin{enumerate}[(i)]
\item The leading coefficients of $\gauss{a}{b}$ seen as a polynomial in $q$ are the possible ways of partitioning
  $b-1$, so sequence A000041 in OEIS. This can be seen in a similar way.
 \item Surely, the lemma is also true when $x$ is not an integer.
But for general $x$ and $q > 1$, the author can only show that (for some constant $C_q$ depending on $q$)
\begin{align*}
  \gauss{a}{b} - q^{bx} \gauss{a-x}{b} \leq (1 + C_q q^{-1}) q^{2x - \lfloor x \rfloor+(b-1)(a-b)-1}.
\end{align*}
This can be seen by combining the proof given here for $x$ an integer with the technique 
used for the proof of Lemma 7 in \cite{Keevash2008}.
The technique in \cite{Keevash2008} on its own only seems to yield $2x+(b-1)(a-b)-1$ in the exponent.
\end{enumerate}
\end{remark}

\subsection{Geometry}

We rely on the existing results on intersecting families and partial spreads of $k$-spaces in $\FF_q^n$.
If $Y$ is the family of all $k$-spaces containing a fixed point, then we call $Y$ a \textit{dictator}.
If $Y$ is the family of all $k$-spaces contained in a fixed hyperplane, then we call $Y$ a \textit{dual dictator}.
Extending work by Hsieh \cite{Hsieh1975} and Frankl and Wilson \cite{Frankl1986}, Newman showed the following \cite{Newm2004}:

\begin{theorem}
  If $n \geq 2k$, then the size of an intersecting family $Y$ of $k$-spaces in $\FF_q^n$ is at most $\gauss{n-1}{k-1}$.
  Equality holds in one of the following two cases:
  \begin{enumerate}[(i)]
   \item the family $Y$ is a dictator,
   \item we have $n=2k$ and the family $Y$ is a dual dictator.
  \end{enumerate}
\end{theorem}

We will use the following simple and well-known facts.
\begin{lemma}\label{lem:int_dict}
  \begin{enumerate}[(i)]
   \item Two dictators intersect in at most $\gauss{n-2}{k-2}$ elements.
   \item A dictator and a dual dictator intersect in at most $\gauss{n-2}{k-1}$ elements.
   \item Let $Y$ be a dictator, or let $Y$ be a dual dictator with $n=2k$. A $k$-space not in $Y$ meets at most $[k]\gauss{n-2}{k-2}$ elements of $Y$.\hfill\qedsymbol
  \end{enumerate} 
\end{lemma}

The following is implied for $n > 2k$ in \cite[Theorem 1.4]{Blokhuis2010}, for large $q$ and $n=2k$ by Blokhuis et al. \cite{Blokhuis2012}, and explicitly shown for $q \geq 4$ and $n=2k$ by the author \cite[Theorem 1.6]{Ihringer2018}.
\begin{theorem}\label{thm:gen_stab}
  Let $n \geq 2k$ and $Y$ is an intersecting family of $k$-spaces in $\FF_q^n$ with $|Y| > 3 [k] \gauss{n-2}{k-2}$. 
  Unless $n=2k$ and $q \in \{ 2,3,4 \}$, or $n=2k+1$ and $q=2$, then $Y$ is contained in a dictator or a dual dictator.
\end{theorem}

Let
\begin{align*}
  z(n, k, q) := \frac{q^{k+r} [n-k-r]}{[k]} + 1.
\end{align*}
A \textit{partial spread} is a set of pairwise disjoint $k$-spaces.
Beutelspacher showed the following \cite{Beutelspacher1975}.
\begin{theorem}
  Let $n = mk + r$ with $0 \leq r < k$. Then the largest partial spread of $k$-spaces of $\FF_q^n$ has size
  $z(n, k, q)$.
\end{theorem}
When $n, k, q$ are clear from the context, we write $z$ instead of $z(n, k, q)$.
While we are not concerned about large $s$, note that this implies $s \leq z$ is an upper bound on $s$
which is in general smaller than the trivial bound of $[n]/[k]$. For instance $z(5, 2, q) = q^3+1$, while
$[5]/[2] = q^3+q + \frac{1}{q+1}$.
We denote a partial spread of size $z$ as a \textit{$z$-spread}.
We will also need the well-known fact that a $k$-space is disjoint to
\begin{align}
  q^{k\ell} \gauss{n-k}{\ell} \label{eq:nmb_disj}
\end{align}
$\ell$-spaces of $\FF_q^n$ \cite[Theorem 3.3]{Hirschfeld1998}.
It follows that if we fix two disjoint $k$-spaces $A$ and $B$, then at least
\begin{align*}
  q^{k^2} \gauss{n-k}{k} - [k] \gauss{n-1}{k-1}
\end{align*}
$k$-spaces are disjoint to both of them. Reason is that $A$ is disjoint to $q^{k^2} \gauss{n-k}{k}$ $k$-spaces.
As $B$ has $[k]$ points and each point lies in $\gauss{n-1}{k-1}$ $k$-spaces, $B$ meets at most $[k] \gauss{n-1}{k-1}$
of these.

Let $n_i$ be the number of $z$-spreads through $i$ fixed, pairwise disjoint $k$-spaces.
An easy double counting argument shows (for instance, see \cite{Blokhuis2018,Rodgers2018}) that
\begin{align*}
  &\frac{n_1}{n_2} = \frac{q^{k^2} \gauss{n-k}{k}}{z-1} = \frac{q^{k^2} [k]}{q^{k+r} [n-k-r]} \gauss{n-k}{k},\\
  &\frac{n_2}{n_3} = \frac{q^{k^2} \gauss{n-k}{k} - [k] \gauss{n-1}{k-1}}{z-2}.
\end{align*}

\section{Proof of the Main Theorem}

In this section we consider a $s$-EM family of $k$-spaces in $\FF_q^n$.
Recall that $n = mk+r$ with $0 \leq r < k$ and that $\ell$ is the integer with 
$[\ell-1] < s \leq [\ell]$.
We assume that $Y$ has size at least
\begin{align*}
  y := s \left( \gauss{n-1}{k-1} - [\ell-1] \gauss{n-2}{k-2} \right).
\end{align*}
If we take $s$ points in an $\ell$-space and let $Y$ be the family of $k$-spaces
which contain at least one of these points, then it is easy to see that $|Y| \geq y$.
Hence, we show Theorem \ref{thm:main} by showing the following stability version of it.

\begin{theorem}\label{thm:main_der}
  Let $n \geq 2k$ and let $Y$ be a $s$-EM family of $k$-spaces in $\FF_q^n$ of size at least $y$.
  If $16 s \leq$ $\min\{$ $q^{\frac{n-k-\ell+2}{3}},$ $q^{\frac{n-k-r}{3}},$ $q^{\frac{n}2-k+1} \}$,
  then $Y$ is the union of $s$ intersecting families.
\end{theorem}

\paragraph*{Assumption} From now on we assume that $16 s \leq \min\{$ $q^{\frac{n-k-\max(r, \ell-2)}{3}},$ $q^{\frac{n}{2}-k+1} \}$ till the end of the section.
Hence, using $16s \leq q^{\frac{n}{2}-k+1}$, we assume that that $n \geq 2k+8$ if $q = 2$, $n \geq 2k+5$ if $q \leq 3$, $n \geq 2k+3$ if $q \leq 4$, $n \geq 2k+3$ if $q \leq 5$, $n \geq 2k+2$ if $q \leq 9$, and $n \geq 2k+1$ if $q \leq 31$ as the theorem does not say anything non-trivial for the excluded cases.
Recall that the first interesting case is $s=2$, so we also assume $s \geq 2$.

\begin{lemma}\label{lem:spread_avg}
  Let $Z$ be a $z$-spread (chosen uniformly and randomly out of all $z$-spreads). Then
  \begin{align*}
    \EE(|Y \cap Z|) > s - 4 \tau s \frac{[k-1][\ell-1]}{[n-1]},
  \end{align*}
  where $\tau = 1$ if $\ell \geq r+2$ and $\tau = q^{r-\ell+2}$ otherwise.
\end{lemma}
\begin{proof}
  Using $n - k - r \geq k \geq 2$ and the limit of the geometric series to bound $[n]/[n-k-r]$, we obtain
  \begin{align*}
    \frac{\gauss{n}{k}}{|Z|} &\leq \frac{[n]}{[k]} \gauss{n-1}{k-1} \cdot \frac{[k]}{q^{k+r} [n-k-r]}\\
    &= \frac{[n]}{q^{k+r} [n-k-r]} \gauss{n-1}{k-1} \leq (1 + \frac{4}{3} q^{k+r-n}) \gauss{n-1}{k-1}.
  \end{align*}
  As $\PGL(n, q)$ acts transitively on $k$-spaces, the average size of the intersection is
  \begin{align*}
    \frac{|Y|\cdot |Z|}{\gauss{n}{k}} &\geq \frac{y} {(1 + \frac{4}{3} q^{k+r-n}) \gauss{n-1}{k-1}}\\
    &\geq s (1 - \frac{4}{3} q^{k+r-n}) \left( 1 - \frac{[k-1][\ell-1]}{[n-1]} \right)\\
    &\geq s - \frac{4}{3} s q^{k+r-n} - s \frac{[k-1][\ell-1]}{[n-1]}.
  \end{align*}
  We have to show that $\frac{4}{3} q^{k+r-n} \leq 3 \tau \frac{[k-1][\ell-1]}{[n-1]}$.
  As $\tau \leq q^{r-\ell+2}$, it suffices to show that
  \begin{align*}
    \frac{9}{4} \left( q^{k+\ell-2} - q^{k-1} - q^{\ell-1} + 1 \right) \geq q^{k+\ell-2} - q^{k+\ell-3} - q^{k+\ell-n-1} + q^{k+\ell-n-2}.
  \end{align*}
  As $q^{k+\ell-n-1} - q^{k+\ell-n-2} > 0$, this is implied by
  \begin{align*}
    \frac{5}{4} q^{k+\ell-2} + q^{k+\ell-3} + \frac{9}{4} \geq \frac{9}{4}\left( q^{k-1} + q^{\ell-1} \right).
  \end{align*}
  Due to monotonicity and $k, \ell, q \geq 2$, we only have to check this inequality for $k=\ell=q=2$.
\end{proof}

From here on, let $\tau$ be as in Lemma \ref{lem:spread_avg}.
For a $k$-space $S$, let $w_S$ denote $\EE(|Y \cap Z|: S \in Z)$ for all $z$-spreads $Z$ which contain $S$.

\begin{corollary}\label{cor:good_spread_ex}
 There exists a $z$-spread $Z$ such that all elements $S \in Y \cap Z$ satisfy $w_S > s - 4\tau s^2 \frac{[k-1][\ell-1]}{[n-1]}$
\end{corollary}
\begin{proof}
  By averaging and Lemma \ref{lem:spread_avg}, we find a $z$-spread $Z$ with $\sum_{S \in Y \cap Z} w_S \geq s(s-4 \tau s \frac{[k-1][\ell-1]}{[n-1]})$.
  We have $w_S \leq s$. The worst case is that $s-1$ elements $S \in Y \cap Z$ have $w_S = s$. Then the remaining element $T$ satisfies
  \begin{align*}
    w_T \geq \sum_{S \in Y \cap Z} w_S - (s-1)s = s - 4\tau s^2 \frac{[k-1][\ell-1]}{[n-1]}.
  \end{align*}
  This shows the claim.
\end{proof}

Let $Y'$ be the set of elements $S \in Y$ such that $\EE(|Y \cap Z|) \geq s - 4\tau s^2 \frac{[k-1][\ell-1]}{[n-1]}$ for all $z$-spreads $Z$ with $S \in Z$.
\begin{lemma}\label{lem:Y_neighs}
  \begin{enumerate}[(i)]
   \item An element $S \in Y$ meets at least \[  \gauss{n-1}{k-1} \left( 1 - 2s q^{\ell+k-n-2} - 2 (s-1) q^{k+r-n} \right) \] elements of $Y$.
   \item For $S, T \in Y'$, there are at most \[ 2 q^{(k-2)(n-k+1)+1} + 12 s q^{(k-2)(n-k+1)}  + 128 \tau s^2 q^{(k-2)(n-k)+\ell-2} \] elements of $Y$ which meet $S$ and $T$.
  \end{enumerate}
\end{lemma}
\begin{proof}
  By double counting $(Z, R)$, where $Z$ is a partial $z$-spread with $R, S \in Z$ with $R$ is disjoint to $S$, we see that $S$ is disjoint to at most $(w_S-1) \frac{n_1}{n_2}$ elements
  of $Y$. Hence, $S$ meets $|Y| - (w_S-1) \frac{n_1}{n_2}$ elements of $Y$.
  
  Similarly, double counting $(Z, R)$, where $Z$ is a partial spread of size $z$ with $S, T \in Z$
  and $R \in Y$ with $R$ is disjoint to $S$ and $T$, shows that $S$ and $T$ are disjoint to at most 
  $(s-2) \frac{n_2}{n_3}$ elements of $Y$. Hence, $S$ and $T$ meet at most
  \[ A := |Y| - (w_S + w_T - 2) \frac{n_1}{n_2} + (s-2) \frac{n_2}{n_3} \]
  elements of $Y$ simultaneously.
  What remains are some tedious calculations. 
  In the case of (i), where we ask for an upper bound, we use $w_S \leq s$.
  Then
  \begin{align*}
    |Y| - (w_S-1) \frac{n_1}{n_2} &\geq y - (s-1) \frac{q^{k^2} \gauss{n-k}{k}}{z-1}\\
    &= y - (s-1) q^{k^2} \frac{[n-k]}{[n-k-r]q^{k+r}} \gauss{n-k-1}{k-1}\\
    &\geq y - (s-1) q^{k^2-k} (1 + 2q^{k+r-n}) \gauss{n-k-1}{k-1}. \\
    \intertext{We continue using $q^{k^2-k} \gauss{n-k-1}{k-1} \leq \gauss{n-1}{k-1}$ and $\gauss{n-1}{k-1} = \frac{[n-1]}{[k-1]} \gauss{n-2}{k-2}$, so}
    |Y| - (w_S-1) \frac{n_1}{n_2}  &\geq \gauss{n-1}{k-1} \left( 1 - s\frac{[\ell-1][k-1]}{[n-1]} - 2(s-1) q^{k+r-n}\right)\\
    &\geq \gauss{n-1}{k-1} \left( 1 - 2s q^{\ell+k-n-2} - 2(s-1) q^{k+r-n} \right).
  \end{align*}

  Set $\delta = 4 \tau s^2 \frac{[k-1][\ell-1]}{[n-1]}$.
  For (ii), we use that $w_S, w_T > s - \delta$.
  We have that
  \begin{align*}
    A &= y - 2(s-1-\delta) \frac{q^{k^2}\gauss{n-k}{k}}{z-1} + (s-2) \frac{\left( q^{k^2}\gauss{n-k}{k} - [k] \gauss{n-1}{k-1} \right)}{z-2}\\
    &= y - s \frac{q^{k^2} \gauss{n-k}{k}}{z-1} + (s-2) \frac{q^{k^2} \gauss{n-k}{k}}{(z-1)(z-2)} - (s-2) \frac{[k] \gauss{n-1}{k-1}}{z-2} + 2\delta \frac{ q^{k^2} \gauss{n-k}{k}}{z-1}\\
    &\leq y - s q^{k^2-k} \gauss{n-k-1}{k-1} + 8s q^{(k-2)(n-k)} - (s-2) q^{(k-2)(n-k+1)+1} + 2\delta q^{k^2} \frac{\gauss{n-k}{k}}{z-1}.
    \intertext{As $\ell \geq 2$, $\tau s \geq q^r \geq 1$, and $\gauss{n-k}{k} \leq \frac{7}{2} q^{k(n-2k)}$, we obtain}
    &8s q^{(k-2)(n-k)} + 2\delta q^{k^2} \frac{\gauss{n-k}{k}}{z-1} \leq 8\tau s^2 q^{(k-2)(n-k)} + 112 \tau s^2 q^{(k-2)(n-k)+\ell-2} \leq  128 \tau s^2 q^{(k-2)(n-k)+\ell-2}.
    \intertext{Now Lemma \ref{lem:high_order_cancellation} together with $y \leq s\gauss{n-1}{k-1}$ shows}
    A &\leq 2 q^{(k-2)(n-k+1)+1} + 12 s q^{(k-2)(n-k+1)} + 128 \tau s^2 q^{(k-2)(n-k)+\ell-2}.
  \end{align*}
  The assertion follows.
\end{proof}

\begin{proof}[Proof of Theorem \ref{thm:main_der}]
  First we show that $Y$ contains $s$ intersecting families $\scrE_1, \ldots, \scrE_s$ such
  that $Y \setminus \bigcup_{i=1}^s \scrE_i$ is small.
  From this we then conclude that $Y \setminus \bigcup_{i=1}^s \scrE_i$ is actually empty.
  
  By Corollary \ref{cor:good_spread_ex}, there exists a $z$-spread $Z$ such that $|Y' \cap Z|=s$.
  Write $\{ S_1, \ldots, S_s\} = Y' \cap Z$. Let $\scrE_i$ denote the set of elements of $Y$ 
  which meet $S_i$ and are disjoint to any $S_j$ with $i \neq j$.
  By Lemma \ref{lem:Y_neighs},
  \begin{align*}
    |\scrE_i| &\geq \gauss{n-1}{k-1} \left( 1 - 2s q^{\ell+k-n-2} - 2(s-1) q^{k+r-n} \right)\\
    & ~~~ - (s-1) \left( 2 q^{(k-2)(n-k+1)+1} + 12 s q^{(k-2)(n-k+1)}  + 128\tau s^2 q^{(k-2)(n-k)+\ell-2} \right).
  \end{align*}
  In the following, we will bound the individual terms of the sum.
  
  Recall that $16 s \leq q^{\frac{n}{2} - k + 1}$, so $2s q^{\ell+k-n-2} \leq \frac{1}{8} q^{\ell-\frac{n}{2}-1}$.
  Particularly, $q^{\ell-2} < s$ implies that $\ell \leq \frac{n}{2} - k + \frac{5}{2}$.
  Hence, as $k \geq 2$, \[ 2s q^{\ell+k-n-2} \leq \frac{1}{8} q^{-k+\frac{3}{2}} < \frac{1}{8}.\]
  
  Next we bound $2(s-1) q^{k+r-n}$ using $16s \leq q^{\frac{n}{2}-k+1}$.
  We have $-\frac{n}{2}+r+1 \leq -k+\frac{r}{2}+1 \leq -\frac{k-1}{2}$. Hence, if $k \geq 3$ or $q \geq 4$, then
  \begin{align*}
    2(s-1) q^{k+r-n} \leq \frac{1}{8} q^{-\frac{k-1}{2}} \leq \frac{1}{16}.
  \end{align*}
  If $q \leq 3$ and $k=2$, then $n \geq 2k+5$ and $r \leq 1$. Hence,
  \begin{align*}
    2(s-1) q^{k+r-n} \leq \frac{1}{8} q^{-\frac{5}{2}} \leq \frac{1}{16}.
  \end{align*}

  We conclude that
  \begin{align*}
    \gauss{n-1}{k-1} \left( 1 - 2s q^{\ell+k-n-2} - 2(s-1) q^{k+r-n} \right) \geq \frac{13}{16} \gauss{n-1}{k-1}.
  \end{align*}
  
  Next we bound the remaining terms of the right hand side.
  As $16s \leq q^{\frac{n}{2}-k+1}$, we have that
  \begin{align*}
   2 (s-1) q^{(k-2)(n-k+1)+1} &\leq \frac{1}{8} q^{(k-2)(n-k+1)+\frac{n}{2}-k+2} \\
   &\leq \frac{1}{8} q^{(k-1)(n-k)-\frac{n}{2}+k} \leq \frac{1}{8} \gauss{n-1}{k-1}.
  \end{align*}
  
  Again, using $16 s \leq q^{\frac{n}{2}-k+1}$, we have that
    \begin{align*}
    12 s(s-1) q^{(k-2)(n-k+1)} &\leq \frac{3}{4} \cdot \frac{1}{16} q^{(k-2)(n-k+1)+n-2k+2} \\
    &= \frac{3}{64} q^{(k-1)(n-k)} \leq \frac{3}{64} \gauss{n-1}{k-1}.
  \end{align*}
  
  We distinguish between $\tau = 1$ and $\tau = q^{r+2-\ell}$.
  If $\tau = 1$, we have, using $16 s \leq q^{\frac{n-k-\ell+2}{3}}$,
  \begin{align*}
    128 (s-1)s^2 q^{(k-2)(n-k)+\ell-2} \leq \frac{1}{32} q^{(k-1)(n-k)} \leq \frac{1}{32} \gauss{n-1}{k-1}.
  \end{align*}
  If $\tau = q^{r+2-\ell}$, we have, using $16 s \leq q^{\frac{n-k-r}{3}}$,
  \begin{align*}
    128 \tau (s-1)s^2 q^{(k-2)(n-k)+\ell-2} = 128 (s-1)s^2 q^{(k-2)(n-k)+r} \leq \frac{1}{32} \gauss{n-1}{k-1}.
  \end{align*}
  Hence,
  \begin{align*}
    |\scrE_i| \geq \frac{39}{64} \gauss{n-1}{k-1}.
  \end{align*}
  We intend to show that $|\scrE_i| > 3 [k] \gauss{n-2}{k-2}$, so that we can apply Theorem \ref{thm:gen_stab}. Therefore, it suffices to show that
  \begin{align*}
    39 [n-1] > 3 \cdot 64 [k][k-1].
  \end{align*}
  This is implied by
  \begin{align*}
    39 (q-1)(q^{n-1}-1) > 192 (q^{2k-1}-1).
  \end{align*}
  If $q \geq 7$ and $n \geq 2k$, then $39 \cdot 6 \geq 192$ shows the inequality.
  If $q \leq 5$, then $n \geq 2k+3$. Hence, $39 (q-1)(q^{n-1}-1) \geq 39 (q^3-1)(q^{2k-1}-1) \geq 192(q^{2k-1}-1)$ shows the inequality.
  Hence, by Theorem \ref{thm:gen_stab}, $\scrE_i$ lies in a unique dictator or dual dictator $\scrE_i'$.
  
  We finish the proof by contradiction. Suppose that there exists a $T \in Y \setminus \bigcup_{i=1}^s \scrE_i'$.
   By Lemma \ref{lem:int_dict} (iii), we do know that at most $[k] \gauss{n-2}{k-2}$ elements of $\scrE_i$ meet $T$.
  First we consider the case that $n>2k$.
  Then, by Lemma \ref{lem:int_dict} (i), $|\scrE_i \cap \scrE_j| \leq \gauss{n-2}{k-2}$ for $i \neq j$.
  Hence, as $16s \leq q^{\frac{n}{2}-k+1}$, we have that
  \begin{align*}
    |\scrE_i| - s \gauss{n-2}{k-2} \geq \frac{39}{64} \gauss{n-1}{k-1} - s \gauss{n-2}{k-2} > 0.
  \end{align*}
  Hence, there exists an element $Z_i$ in each $\scrE_i \setminus \bigcup_{j \neq i} \scrE_j$ which is disjoint to $T$.
  Thus $\{ Z_1, \ldots, Z_s, T \}$ is a subset of $s+1$ pairwise disjoint elements in $Y$, a contradiction.
  
  For $n=2k$, by Lemma \ref{lem:int_dict} (ii), we can only guarantee that $|\scrE_i \cap \scrE_j| \leq \gauss{n-2}{k-1}$ for $i \neq j$.
  As $16s \leq q$ and $k \geq 2$, our estimate is
  \begin{align*}
    |\scrE_i| - s \gauss{n-2}{k-1} \geq \frac{39}{64} \gauss{n-1}{k-1} - s \gauss{n-2}{k-1} > 0.
  \end{align*}
  As before, this is a contradiction.
\end{proof}

\section{Cameron-Liebler Classes} \label{sec:cl}

Cameron-Liebler classes of $k$-spaces on $\FF_q^n$, which the author often refers to as Boolean degree $1$ functions of $k$-spaces on $\FF_q^n$ \cite{Filmus2018},
are well-investigated objects \cite{Blokhuis2018,Filmus2018,Rodgers2018}. In particular for the case $n=4$ and $k=2$ where they are known as Cameron-Liebler line classes. When $k$ divides $n$ (so a $z$-spread is simply a spread), one particular property of Cameron-Liebler classes is that their size is $s\gauss{n-1}{k-1}$ for some integer $s$ and that every spread intersects them in exactly $s$ elements \cite{Blokhuis2018}.
In the following, define $s$ by $|Y| = s \gauss{n-1}{k-1}$, even if $k$ does not divide $n$.
Theorem 4.9 in \cite{Blokhuis2018} claims a result similar to Theorem \ref{thm:main_cl}.
A minor, but sadly consequential sign-error in Lemma 4.6 of \cite{Blokhuis2018} makes the proof of Theorem 4.9 false in the stated form.
Below Lemma \ref{lem:fix_joz} provides a fix for Lemma 4.6 of \cite{Blokhuis2018}. We use this to show Theorem \ref{thm:main_cl}. We do not have to show anything for $n=2k$, as this case
is implied by Theorem \ref{thm:metsch_n_eq_2k}. We also do not have to show anything for $q \in \{ 2, 3, 4, 5 \}$
as in this case all Cameron-Liebler classes were classified in \cite{Filmus2018} for $(n, k) \neq (4, 2)$.

\begin{lemma}\label{lem:fix_joz}
  Let $n \geq 2k+1$ and $q \geq 7$.
  Let $Y$ be a Cameron-Liebler class of $k$-spaces on $\FF_q^n$ of size $s \gauss{n-1}{k-1}$. 
  If $s^3 \leq q^{n-2k-\tilde{r}+1 }$, where $n = \tilde{m} k - \tilde{r}$ with $0 \leq \tilde{r} < k$, then $Y$ contains at most $s$ pairwise disjoint $k$-spaces.
\end{lemma}
\begin{proof}
  As shown in \cite[Lemma 4.6]{Blokhuis2018}, this is equivalent to
  \begin{align*}
    \frac{(1-\lfloor s \rfloor)s \lfloor s \rfloor}{2} \gauss{n-1}{k-1} + (s-1)(\lfloor s \rfloor^2 - 1) q^{k^2-k} \gauss{n-k-1}{k-1}\\
    > \frac{(s-2) (\lfloor s \rfloor + 1) \lfloor s \rfloor}{2} W_\Sigma,
  \end{align*}
  where $W_\Sigma$ denotes the number of $k$-spaces through a point disjoint to two fixed, disjoint $k$-spaces. Note that this part of \cite[Lemma 4.6]{Blokhuis2018} requires that $n \geq 2k+1$, but not $n \geq 3k$ as required there.
  
  The coefficient of the first term is negative, so (this is the mistake in \cite[Lemma 4.6]{Blokhuis2018}), we can obtain 
  a sufficient condition by substituting $\gauss{n-1}{k-1}$ by the upper bound from Lemma \ref{lem:high_order_cancellation} for $q \geq 7$.
  We will bound $W_\Sigma$ with Equation \eqref{eq:nmb_disj}.
  Hence, it suffices that
  \begin{align*}
    \frac{(1-\lfloor s \rfloor)s \lfloor s \rfloor}{2} \left( q^{k^2-k} \gauss{n-k-1}{k-1} + \frac32 q^{1+(k-2)(n-k+1)} \right)\\ + (s-1)(\lfloor s \rfloor^2 - 1) q^{k^2-k} \gauss{n-k-1}{k-1}\\
    > \frac{(s-2) (\lfloor s \rfloor + 1) \lfloor s \rfloor}{2} q^{k^2-k} \gauss{n-k-1}{k-1}.
  \end{align*}
  Rearranging yields
  \begin{align*}
    8 \left( \lfloor s \rfloor - s + 1 \right) q^{k^2-k} \gauss{n-k-1}{k-1} > 6 \lfloor s \rfloor (\lfloor s \rfloor - 1) s q^{1+(k-2)(n-k+1)}.
  \end{align*}
  Hence, it suffices to guarantee
  \begin{align*}
   8 \left( \lfloor s \rfloor - s + 1 \right) q^{n-2k+1} > 6s^3.
  \end{align*}
  It is shown in \cite[Theorem 2.9.4]{Blokhuis2018} that $s[k]$ divisible by $[n]$. Hence, 
  $\lfloor s \rfloor - s + 1$ is at least $(q-1) q^{-\tilde{r}-1}$. The assertion follows using $q\geq 7$.
\end{proof}

Hence, using Theorem \ref{thm:main}, we obtain Theorem \ref{thm:main_cl}. 
We do not need the conditions $16s \leq q^{\frac{n}{2}-k+1}$ and $16s \leq q^{\frac{n-k-r}{3}}$
in Theorem \ref{thm:main_cl} as these are always implied by one of the other two bounds on $s$.

\section{Almost Counterexamples and Future Work} \label{sec:almost_cnt_ex}

One objective of this project was to find counterexamples to the natural Conjecture \ref{conj:em_vs}.
Obviously, we did not achieve this goal and it is left to future work. For $(n, k) = (4, 2)$, we have $(q^2+1)(q^2+q+1)$ lines.
The trivial upper bound is $s(q^2+q+1)$. 
By combining intersecting families, it is easy to obtain examples of size $s(q^2+q)+2$ for $s \leq 2q$.
This number is still very close to the trivial bound, so it seems unreasonable to find counterexamples in this range.
If we limit ourselves to $s \leq \frac{q^2+1}{2}$, so we take at most half of all lines, then
maybe the first plausible parameter to look at is $q=5$ with $s=11$.

Here we will provide one construction which shows that it is hard to extend the range of $s$ Theorem 
\ref{thm:main} significantly. The examples are limited to $(n, k) = (4, 2)$ for the sake of clarity.
We take an elliptic quadric $\scrQ$ in $\FF_q^4$. This consists of $q^2+1$ points, no three of which are collinear.
A line which contains two points of $\scrQ$ is called a secant. Let $Y$ be the family of all secants.
Clearly, $|Y| = \binom{q^2+1}{2} = \frac{q^2}{2} (q^2+1)$ and, if $q$ even, then $Y$ contains at most $\frac{q^2}{2}$ pairwise disjoint secants.
Hence, $s=\frac{q^2}{2}$. For sufficiently large $q$, it is not too hard to find a union $Y'$ of $\frac{q^2}{2}$ intersecting families 
with\footnote{Fix a line $\ell$ and a plane $\pi$ with $\ell$. Let $\scrP$ a set of $\frac{q^2}{2} - q$ points in $\pi \setminus \ell$. 
Let $Y'$ be the union of the set of lines in planes through $\ell$ and the set of all lines which contain a point of $\scrP$.
Then $|Y'| = q(q^2+q)+1 + (\frac{q^2}{2} - q) q^2 + q + 1 = \frac{q^4}{2} + q^2 + q + 2$.}
$|Y'| = \frac{q^2}{2} \cdot q^2 + q^2+q+2$. Here $|Y'| - |Y| = \frac{q^2}{2} + q + 2$.

There are several other similar constructions using quadric curves and related objects such as hyperovals, but we could never extend
them in a way that it disproves Conjecture \ref{conj:em_vs}. We could also not adapt any of the many constructions
for non-trivial Cameron-Liebler line classes for $(n, k) = (4, 2)$ to obtain such a counterexample. Our search here was surely very incomplete as for instance \cite{Gavrilyuk2018} and \cite{Rodgers2012} show that 
there are many such examples.

Furthermore, there are other classical geometrical structures for which the Erd\H{o}s Matching Conjecture
might be interesting. For instance, one can easily deduce the following using the same methods as in Theorem \ref{thm:main} for some universal constant $C$.
\begin{theorem}
  Let $n \geq 2k$ and $Y$ be an $s$-EM family of $k$-spaces in $\AG(n, q)$.
  If $C s \leq$ $\min\{$ $q^{\frac{n-k-\ell+2}{3}},$ $q^{\frac{n-k-r}{3}},$ $q^{\frac{n}2-k+1} \}$,
  then $Y$ is the union of $s$ intersecting families.
\end{theorem}
Here improvements on this bound might be easier compared to the investigated case
as spreads always exist.
Similarly, $k \times (n-k)$-bilinear forms over $\FF_q$ can be seen as the set of $k$-spaces 
which are disjoint to a fixed $(n-k)$-space \cite[\S9.5]{Brouwer1989}. Again, a analogous result is
easy to show.
\begin{theorem}
  Let $n \geq 2k$ and $Y$ be an $s$-EM family of $k \times (n-k)$-bilinear forms over $\FF_q$.
  If $C s \leq$ $\min\{$ $q^{\frac{n-k-\ell+2}{3}},$ $q^{\frac{n-k-r}{3}},$ $q^{\frac{n}2-k+1} \}$,
  then $Y$ is the union of $s$ intersecting families.
\end{theorem}
The trivial bound here is $s \left( \gauss{n-1}{k-1} - \gauss{n-2}{k-2} \right)$ 
(instead of $s\gauss{n-1}{k-1}$ for vector spaces) which can 
be easily obtained for all $s \leq [k]$.
It might be easier to find counterexamples to the natural variation of Conjecture \ref{conj:em_vs}
in affine spaces or bilinear forms.

Recall that the statement of Theorem \ref{thm:main_cl} is empty for 
$2k < n < \frac{5}{2} k$. We believe that this range can be covered by using a better estimate than Lemma \ref{lem:high_order_cancellation} and a better upper bound on $|W_\Sigma|$ in Lemma \ref{lem:fix_joz}. More precisely, in these lemmas we compare Gaussian coefficients by their largest terms 
(seen as polynomials in $q$), while more terms cancel. Indeed, this happens for the $n=2k$ proof in \cite{Metsch2017}.

\section{Acknowledgements}

The author thanks the referees for their detailed comments.
The author is supported by a postdoctoral fellowship of the Research Foundation - Flanders (FWO).



%  \bibliographystyle{elsarticle-num} 
%  \bibliography{../../jabref}


\end{document}